\documentclass[11pt, a4paper]{amsart}

\usepackage{amsfonts,amsmath,amssymb, amscd,fullpage}
\usepackage[all]{xy}

\newtheorem{theorem}{Theorem}[section]
\newtheorem{lemma}[theorem]{Lemma}
\newtheorem{proposition}[theorem]{Proposition}
\newtheorem{corollary}[theorem]{Corollary}

\theoremstyle{definition}
\newtheorem{definition}[theorem]{Definition}

\theoremstyle{remark}
\newtheorem{rem}[theorem]{Remark}

\newtheorem{example}[theorem]{Example}

\newcommand\pf{\begin{proof}}
\newcommand\epf{\end{proof}}

\numberwithin{equation}{section}

\title{Hopf algebras having a dense big cell}

\author{Julien Bichon and Simon Riche}
\address{
Laboratoire de Math\'ematiques (UMR 6620), CNRS,
Universit\'e Blaise Pascal,
Complexe universitaire des C\'ezeaux,
63171~Aubi\`ere Cedex, France}
\email{Julien.Bichon@math.univ-bpclermont.fr, Simon.Riche@math.univ-bpclermont.fr}

\thanks{S.R. was supported by ANR Grants No.~ANR-09-JCJC-0102-01 and ANR-2010-BLAN-110-02.}


\begin{document}

\begin{abstract}
We discuss some axioms that ensure that a Hopf algebra
has its simple comodules classified using
an analogue of  the Borel--Weil construction.
More precisely we show that a Hopf algebra having a dense big cell
satisfies  the above requirement. This method
has its roots in the work of Parshall and Wang in the case of $q$-deformed
quantum groups ${\rm GL}$ and ${\rm SL}$. 
Here we examine the example of universal cosovereign Hopf algebras, for which
the weight group is the free group on two generators.
\end{abstract}

\maketitle

\section*{Introduction}

Let $G$ be connected reductive algebraic group (over a fixed algebraically closed field $k$), and 
let $T \subset B$ be a maximal torus and a Borel subgroup in $G$.
It is well known that there is a bijection between
isomorphism classes of simple $G$-modules and dominant weights of $T$. 
This correspondence is realized using the celebrated Borel--Weil construction, that goes as follows. 
Let $\lambda \in
\Lambda = X^*(T)=X^*(B)$ (these are the character groups of $T$ and $B$), 
and consider the induced $G$-module ${\rm Ind}_B^G(\lambda)$. Then
 ${\rm Ind}_B^G(\lambda)\not =\{0\}$ if and only if $\lambda$ is dominant, in which
 case ${\rm Ind}_B^G(\lambda)$ contains a unique simple 
 $G$-module, denoted $L(\lambda)$. The map defined in this way
 \[
 \left\{
 \begin{array}{ccc}
  \Lambda_+ & \longrightarrow & {\rm Irr}(G)\\
  \lambda &\longmapsto & L(\lambda)
 \end{array}
 \right.
 \]
 is then a bijection. See e.g.~\cite{jan}.

The aim of this paper is to discuss some axioms on a quantum group
that ensure that its simple representations are classified using
the above method. A quantum group  is understood here to
be the  dual object of a Hopf algebra, the latter playing
the role of a function algebra. 
The most popular quantum groups are the $q$-deformed Lie groups 
introduced independently by Drinfeld \cite{dri}, Jimbo \cite{jim} and Woronowicz \cite{wo1}
(see e.g.~the textbooks \cite{bg,jan2,ks} for these objects).
However there are other very interesting examples,
introduced e.g.~in the papers \cite{dvl,wa1, vdw, wa, basp}, related to free products
and free probability (see for example \cite{ba, bb,bc,kosp,bcs} for these kinds of developments),
that do not fit
into the $q$-deformation scheme.

The representation theory of the $q$-deformed quantum groups ${\rm GL}_q(n)$ and ${\rm SL}_q(n)$
was worked out along the classical lines of the Borel--Weil construction
by Parshall and Wang \cite{pw}. In their work, a crucial role is played
by the existence of the analogue of the dense big cell of a reductive
group (the big cell of the reductive group $G$ is the dense open subset $BB'$ for another well-chosen
Borel subgroup $B'$).

In this paper we propose an axiomatisation, at a Hopf algebra level,
of the notion of dense big cell (we do not give the technical definition in this introduction).
We show that the simple comodules of a Hopf algebra having a dense 
big cell are classified in a completely similar manner
to the reductive group case, by means of the Borel--Weil construction.
We hope that this will provide a useful general tool in the representation theory of
 quantum groups. 

 We do not claim that all Hopf algebras fit into this framework, but 
 we believe that after suitable reductions such as
 tensor equivalences, the simple comodules of a wide class of Hopf algebras
 can be descrided in this way. 
 
  We have borrowed many important ideas and arguments
to Parshall and Wang \cite{pw}, but at some points
some of their arguments had to be modified. This is due to the fact that
in general the ``weight group'' $\Lambda$ is non commutative, and 
we could not use a natural order on this group (and hence no natural
notion of a dominant weight).  So what is probably
missing here to complete the picture is a general theory of some kind of 
non commutative root systems, which certainly would require more axioms than the ones we have used.
 
 As an illustrative example, we study the universal cosovereign Hopf algebras \cite{bi1},
 an algebraic counterpart and generalization of Van Daele and Wang's universal
 compact quantum groups \cite{vdw}. In view of their universal property, the universal
 cosovereign Hopf algebras might be considered as some analogues of the general linear
 groups in quantum group theory. 
 As shown in \cite{bi2}, 
 after suitable tensor equivalence reductions, it is enough
 to study a family of Hopf algebras $H(q)$, $q \in k^\times$. We show that the Hopf algebras $H(q)$
 have a dense big cell, with $\mathbb F_2$, the free group on two generators, as a weight group.
 The simple $H(q)$-comodules are then classified by means of the Borel--Weil construction, and the monoid of dominant weights is the free monoid on two generators.
 The classification is characteristic free and does not depend on $q$, and therefore
 generalizes earlier results in the cosemisimple case \cite{ba,bi2}. Note that the combinatorial description of the Grothendieck ring of the category of comodules over the universal cosovereign Hopf algebras has already been given by Chirvasitu \cite{chi}. Our alternative approach  has the merit to provide explicit descriptions for the simple comodules. The problem of finding explicit models for the irreducible representations of universal quantum groups, in connection with Borel--Weil theory, was raised in \cite{wabcp}.
   
The paper is organized as follows. Section 1 consists of preliminaries. In Section 2 we define the Borel--Weil property and dense big cells for Hopf algebras. Section 3 is devoted to the proof of the fact that a Hopf algebra having a dense big cell has  the Borel--Weil property, i.e.~its simple comodules are classified by the Borel--Weil construction. 
In Section \ref{sec:SL2} we study the example of $\mathcal{O}(\mathrm{SL}_q(2))$.
In Sections 5--6 we develop some general tools to prove the existence of a dense big cell or to classify simple comodules,
related to free products and Hopf subalgebras. The results of these sections are used in Section 7 to study our main example, the universal cosovereign Hopf algebras. The existence of a dense big cell is shown for $H(q)$, the weight group being the free group on two generators, and explicit models for the simple comodules are described in general. The last Section 8 deals with the behaviour of dense big cells under 2-cocycle deformations, and we remark that the 2-cocycle deformation of a Hopf algebra having a dense big cell does not necessarily still have a dense big cell.

\textbf{Acknowledgements.} We warmly thank the referee for his careful reading of the manuscript, and for pertinent remarks and hints which improved several results in the paper (Propositions 1.1, 5.1, 6.2), removing unnecessary technical assumptions.

\section{Preliminaries}

We begin by  recalling the basic notions needed in the paper.

We work in general over an algebraically closed field $k$, set $\otimes=\otimes_k$.
We assume that the reader has  some familiarity with Hopf algebras, for which
the textbook \cite{mo}  is convenient. Our terminology and notation are the standard ones: in particular, for a Hopf algebra, $\Delta$, $\varepsilon$ and $S$ denote the comultiplication, counit and antipode, respectively.  

\subsection{Comodules over a Hopf algebra.} Let $H$ be a Hopf algebra.
An $H$-comodule is a vector space $V$ endowed with a linear map
$\alpha : V \longrightarrow V \otimes H$ such that
$$(\alpha \otimes {\rm id_H}) \circ \alpha = ({\rm id}_V \otimes \Delta) \circ \alpha \quad
{\rm and} \quad ({\rm id}_V \otimes \varepsilon) \circ \alpha= {\rm id}_V. $$
If $V$ is finite-dimensional with basis $v_1, \ldots , v_n$, then we have
$$\alpha(v_i)= \sum_{j=1}^n v_j \otimes x_{ji}$$
for elements $x_{ij} \in H$ satisfying
$$\Delta(x_{ij})= \sum_{k=1}^n x_{ik}\otimes x_{kj}, \quad \varepsilon(x_{ij}) =\delta_{ij}.$$
Such a matrix $x=(x_{ij}) \in M_n(H)$ is said to be multiplicative. 
Its entries are called the coefficients of $V$: they generate
a subcoalgebra $H(V)$ of $H$. The comodule $V$ is simple if and only if the
elements $x_{ij}$ are linearly independent in $H$, which is also equivalent 
to saying that $\dim(H(V)) = \dim(V)^2$. 
The matrix $x$ entirely determines
the comodule $V$, and conversely, a multiplicative matrix $x \in M_n(H)$
defines a comodule structure on the vector space $k^n$.

The motivating example is $H = \mathcal O(G)$, the Hopf algebra of polynomial functions on 
a linear $k$-algebraic group $G$: in this case the category of finite dimensional
$H$-comodules is equivalent the category of algebraic representations of $G$.

A simple $H$-comodule is necessarily finite-dimensional.
The set of isomorphism classes of simple $H$-comodules is denoted ${\rm Irr}(H)$, and if $V$ is a simple $H$-comodule, we denote by $[V]\in {\rm Irr}(H)$ its isomorphism class. 
The coradical of $H$, denoted $H_0$, is the (direct) sum of the simple subcoalgebras of $H$. Equivalently
$H_0= \oplus_{[V] \in {\rm Irr}(H)}H(V)$.

The one-dimensional $H$-comodules correspond to the group-like elements of $H$:
$${\rm Gr}(H) = \{g \in H \ | \ \Delta(g)= g \otimes g, \ \varepsilon(g)=1\};$$
we denote by $k_g$ the comodule associated with $g \in {\rm Gr}(H)$. Note that
${\rm Gr}(H)$ is a group for the multiplication of $H$.

We end the subsection by a result that turns out to be useful to prove that all the simple comodules have been determined. Similarly to the other results in this preliminary section, it is probably well known (it is at least known in the cosemisimple case), but in lack of suitable reference, we sketch a proof.

The result uses the formalism of the Grothendieck ring. Given a Hopf algebra $H$, let us say that its Grothendieck ring, denoted $K(H)$, is the Grothendieck ring of the tensor category of finite-dimensional comodules. Here we use the version of the Grothendieck ring for which the isomorphism classes of simple objects form a basis. Again, if $V$ is an $H$-comodule, its isomorphism class (in $K(H)$) is denoted $[V]$.

\begin{proposition}\label{completeset}
 Let $\mathcal J \subset {\rm Irr}(H)$. Assume that the following conditions hold.
\begin{enumerate}
 \item $\forall [V], [W] \in \mathcal J$, in $K(H)$ we have $[V] \cdot [W] = 
\sum_{i=1}^m [Z_i]$ for some $[Z_1], \ldots , [Z_m]  \in \mathcal J$.
\item $H_0$, the coradical of $H$, is contained in the subalgebra generated by the coalgebras $H(V)$, $[V] \in \mathcal J$.
\end{enumerate}
Then $\mathcal J = {\rm Irr}(H)$.
\end{proposition}

\begin{proof}
 The second assumption ensures that each simple $H$-comodule $X$ is isomorphic to a subquotient of a direct sum of tensor products of elements of $\mathcal J$ (see the proof of Proposition \ref{pointed} for a similar argument). The first assumption and the fact that the isomorphism classes of simple comodules form a basis of $K(H)$ now show that $[X] \in \mathcal J$.
\end{proof}

\subsection{Pointed Hopf algebras} 
A Hopf algebra $H$ is said to be \emph{pointed}
if all its simple comodules are  one-dimensional. Thus if $H$ is pointed we have
${\rm Irr}(H) = {\rm Gr}(H)$.
Here is a list of  important examples:
\begin{enumerate}
\item $k\Gamma$, the group algebra of a discrete group $\Gamma$;
 \item $\mathcal O(G)$, where $G$ is a connected solvable algebraic group
 (by the Lie--Kolchin theorem);
 \item $U(\mathfrak g)$, where $\mathfrak g$ is a Lie algebra;
 \item $U_q(\mathfrak g)$, the Drinfeld--Jimbo quantized algebra of a Kac--Moody algebra $\mathfrak g$.
\end{enumerate}
Observe also that a Hopf subalgebra of a pointed Hopf algebra is pointed.

The classification problem for finite-dimensional pointed Hopf algebras
has received much attention in recent years, see \cite{as}.
We wish to prove a few basic results concerning pointed Hopf algebras. 
All these results are probably well known.

Let $H$ be a Hopf algebra and let $V$ be a finite dimensional
$H$-comodule. An $H$-flag on $V$ is an increasing sequence of subcomodules
$$\{0\} = V_0 \subset V_1 \subset \cdots \subset V_{m-1} \subset V_m=V$$
with $\dim (V_i/V_{i-1})= 1$, $i=1, \ldots ,m$.    
Then $H$ is pointed if and only if any finite-dimensional $H$-comodule admits an $H$-flag.
(The proof is done by induction on the dimension of the comodules, the
details are left to the reader.) In matrix version, this remark has the following form. 

\begin{lemma}
 Let $H$ be a Hopf algebra. The following assertions are equivalent:
 \begin{enumerate}
  \item $H$ is pointed;
  \item  for any multiplicative matrix  $x \in M_n(H)$, there exists
  a  matrix $F \in {\rm GL}(n,k)$ such that the (multiplicative) matrix $FxF^{-1}$ is upper triangular;
  \item for any multiplicative matrix  $x \in M_n(H)$, there exists
  a matrix $F \in {\rm GL}(n,k)$ such that the (multiplicative)
  matrix $FxF^{-1}$ is lower triangular.
  \end{enumerate}
\end{lemma}

The following lemma is used in the next section.

\begin{lemma}\label{dependent}
 Let $H$ be a pointed Hopf algebra, and let $x =(x_{ij})\in M_n(H)$ be a multiplicative matrix (with $n \geq 2$).
 The elements
 $$x_{11}, x_{21}, \ldots , x_{n1}, x_{12}, \ldots , x_{n2}$$
 are linearly dependent in $H$.
\end{lemma}

\begin{proof}
By the previous lemma, there exists
$F=(f_{ij}) \in {\rm GL}(n,k)$  such that the matrix $FxF^{-1}=y =(y_{ij})$ is upper triangular.
So we have $Fx = y F$ and in particular
$$\sum_k f_{nk}x_{k1} = f_{n1}y_{nn} \quad {\rm and} \quad \sum_kf_{nk}x_{k2}=f_{n2}y_{nn}.$$
If $f_{n1}=0$ or $f_{n2}=0$ we are done, and otherwise we have
$$\sum_k f_{n1}^{-1}f_{nk}x_{k1} = \sum_kf_{n2}^{-1}f_{nk}x_{k2}$$
which proves our result.
\end{proof}

The following result is  useful to show that a Hopf algebra is pointed.

\begin{proposition}\label{pointed}
 Let $H$ be a Hopf algebra generated by the entries of 
 a family of upper or lower triangular multiplicative matrices.
 Then $H$ is pointed, and the group ${\rm Gr}(H)$ is generated by the diagonal entries of the previous
 matrices.
\end{proposition}

\begin{proof}
 We begin with a general remark: if $H$ is a Hopf algebra, then the property 
 of having an $H$-flag for $H$-comodules is stable under
 \begin{enumerate}
  \item forming tensor products and direct sums,
  \item taking duals,
  \item taking subobjects, quotients and subquotients.
 \end{enumerate}

Now let $H$ be a Hopf algebra generated by the coefficients
of a family $(u_i)_{i \in I}$ of
upper or lower triangular multiplicative matrices. Note that if $U_i$ denotes the $H$-comodule associated with $u_i$, then $U_i$ admits an $H$-flag since $u_i$ is either upper triangular or lower triangular.
$H$ is generated as an algebra by the coefficients of the matrices $u_i$ and $S(u_i)$. Hence
if $V$ is an $H$-comodule, then the comultiplication $\Delta$ realizes $V$ as a subcomodule of a direct sum of subcomodules $H_1, \cdots, H_d$ of $H$ where each $H_j$ is a subquotient of a direct sum of tensor products of comodules of the form $U_i$ or $U_i^*$. Then it follows from the remark at the beginning of the proof that $V$ admits an $H$-flag. 

If $V$ is simple (hence $1$-dimensional), then by the same arguments $V$ is a composition factor of a tensor product of comodules of the form $U_i$ or $U_i^*$. The group-likes corresponding to such composition factors clearly belong to the group generated by the diagonal
entries of the $u_i$'s, which finishes the proof.
%
%
\end{proof}

\begin{example}
 Let $\mathcal O(B_q)$ be the Hopf algebra of the $q$-deformed
 quantum group of lower triangular matrices, considered by Parshall and Wang \cite{pw}.
 It is shown in \cite[Theorem 6.5.2]{pw} that $\mathcal O(B_q)$ is pointed. 
 This can also be deduced from Proposition \ref{pointed}.
\end{example}

\subsection{Restriction and induction of comodules}
We now recall a few basic facts concerning restriction and induction of comodules.
This material is presented in greater detail in \cite{pw}.

Let $f : H \longrightarrow L$ be a Hopf algebra map and let
$V= (V, \alpha)$ be an $H$-comodule. Then the linear map
$\alpha'= ({\rm id}_V \otimes f) \circ \alpha : V \longrightarrow V \otimes L$
endows $V$ with an $L$-comodule structure, denoted $f_*(V)$ or simply again $V$ if no confusion
can arise.  This comodule is the restriction of the $H$-comodule $V$ to $L$.

When the $L$-comodule $f_*(V)$ has a one-dimensional subcomodule, we say that it has 
an $L$-stable line. If $L$ is pointed, then $f_*(V)$ always has an $L$-stable line, by the discussion in the previous subsection. 
The following definition will be convenient.

\begin{definition}
 Let $f : H \longrightarrow B$ be a Hopf algebra map with $B$ pointed and let
$V= (V, \alpha)$ be an $H$-comodule. A \textbf{$B$-weight} for $V$ is an element
$g \in {\rm Gr}(B)$ for which there exists a non zero element $v \in V$ such that
$({\rm id}_V \otimes f) \circ \alpha(v) = v \otimes g$,
the subspace $kv$ being then a $B$-stable line. We say that a $B$-weight is a \textbf{highest $B$-weight}
if $V$ has a unique $B$-stable line.
\end{definition}

We now come to the induction functor (i.e.~the right adjoint to the functor $f_*$). Again let $f : H \longrightarrow L$ be a Hopf algebra map, and let
$W= (W, \beta)$ be an $L$-comodule. The induced $H$-comodule, denoted 
${\rm Ind}_{L}^H(W)$, is the cotensor product $W \square_L H$ (see \cite{mo}). 
We do not recall the general definition since we only consider the case when
$W=k_g$ for some  $g \in {\rm Gr}(L)$. 
In this case we write  ${\rm Ind}_{L}^H(g):={\rm Ind}_{L}^H(k_g)$, and we have
\[
{\rm Ind}_{L}^H(g) = \{ x \in H \ | \ (f\otimes {\rm id_H}) \circ \Delta(x) =
g \otimes x\}\\
=
\{ x \in H \ | \ f(x_{(1)})\otimes x_{(2)} =
g \otimes x\}
\]
where we have used Sweedler's notation $\Delta(x)= x_{(1)} \otimes x_{(2)}$ in the second equality.

\section{Borel--Weil data and dense big cells}

We have now enough material to introduce in this section
the main definitions of the paper: Borel--Weil data and dense big cells.

\subsection{Borel--Weil data}
 We begin by proposing 
an axiomatic definition  for the Hopf algebras whose simple comodules
are described by an analogue of the Borel--Weil construction.

\begin{definition}
A \textbf{Borel--Weil datum} consists of a triple $(H,B,\pi)$ where  $H,B$ are Hopf algebras and $\pi : H \rightarrow B$ is a surjective Hopf algebra map, satisfying the following axioms.
\begin{enumerate}
\item The Hopf algebra $B$ is pointed.
\item For any $\lambda \in \Lambda:={\rm Gr}(B)$, the $H$-comodule ${\rm Ind}_B^H(\lambda)$
is either $\{0\}$ or contains a unique simple $H$-comodule, denoted $L(\lambda)$. 
\item Let $\Lambda_+ = \{ \lambda \in \Lambda \ | \ {\rm Ind}_B^H(\lambda) \not=\{0\}\}$. 
The map
\[
\left\{
\begin{array}{ccc}
\Lambda_+ &\longrightarrow & {\rm Irr}(H) \\
\lambda  &\longmapsto & [L(\lambda)]
\end{array}
\right.
\]
is a bijection, and $\Lambda_+$ is a submonoid of $\Lambda$.
\end{enumerate}
The group $\Lambda$ is called the \textbf{weight group} of the Borel--Weil datum $(H,B,\pi)$, and the weights  $\lambda \in \Lambda_+$ are said to be \textbf{dominant}. If the morphism $\pi$ is understood, a Borel--Weil datum $(H,B,\pi)$ is simply denoted $(H,B)$.
\end{definition}


\begin{rem}\label{submonoid}
If $H$ is an integral domain, then it is automatic that
 $\Lambda_+$ is a submonoid of  $\Lambda$.
Indeed it is clear that $1 \in \Lambda_+$.
 Let $\lambda, \mu \in \Lambda_+$: there exists non zero elements $x,y \in H$
 such that 
 $$\pi(x_{(1)}) \otimes x_{(2)} = \lambda \otimes x \quad {\rm and} \quad
 \pi(y_{(1)}) \otimes y_{(2)} = \mu \otimes y.$$
 Then
$$ \pi((xy)_{(1)}) \otimes (xy)_{(2)} =\pi(x_{(1)})\pi(y_{(1)}) \otimes x_{(2)} y_{(2)}
=  \lambda\mu \otimes xy
$$ 
and hence ${\rm Ind}_B^H(\lambda\mu)\not = \{0\}$.
\end{rem}

\begin{example}
\begin{enumerate}
\item
 If $B$ is a pointed Hopf algebra, then $(B,B,{\rm id}_B)$ is a Borel--Weil datum.
 \item
Our motivating example is the following: if $G$ is a connected reductive algebraic group and $B \subset G$ a Borel subgroup, then $(\mathcal O(G), \mathcal O(B))$
is a Borel--Weil datum.
\end{enumerate}
\end{example}


\begin{definition}
A  Hopf algebra $H$ is said to have the \textbf{Borel--Weil property}
when there exists a pointed Hopf algebra $B$ and a surjective Hopf algebra map  $\pi : H \rightarrow B$
 such that $(H,B, \pi)$ is a Borel--Weil
datum.
\end{definition}


Parshall and Wang have shown in \cite{pw} that the Hopf algebras $\mathcal O({\rm GL}_q(n))$ and
 $\mathcal O({\rm SL}_q(n))$ have the Borel--Weil property, thus describing the simple representations
 of the corresponding quantum groups.
 This is the first example of a ``noncommutative Borel--Weil situation,'' and one of the main sources of motivation for the present work.

\subsection{Dense big cells}
A key technical ingredient in \cite{pw} to prove that the Hopf algebras $\mathcal O({\rm GL}_q(n))$ and
 $\mathcal O({\rm SL}_q(n))$ have the Borel--Weil property in the presence of an analogue
 of the dense big cell of a reductive group.
 Recall that if $G$ is a reductive algebraic group, then 
 there exists Borel subgroups $B$, $B'$ of $G$ such that the ``big cell''
 $BB'$ is dense in $G$ (see \cite{bor}).
 Here is  the corresponding axiomatization. 

\begin{definition}
 Let $H$ be a Hopf algebra.
  A \textbf{dense big cell} for $H$ consists of a triple $(B,B',\Lambda)$ formed by two pointed Hopf algebras $B,B'$ and a discrete group $\Lambda$ (called the \textbf{weight group}) together with surjective Hopf algebra maps 
\[
\xymatrix@C=1.5cm{
H \ar@{->>}[r]^-{\pi} & B \ar@{->>}[r]^-{\psi} & k\Lambda
}
\quad \text{and} \quad
\xymatrix@C=1.5cm{
H \ar@{->>}[r]^-{\pi'} & B' \ar@{->>}[r]^-{\psi'} & k\Lambda
}
\]
satisfying the following axioms.
\begin{enumerate}
\item We have $\psi \circ \pi = \psi' \circ\pi'$,
and $\psi$, $\psi'$ induce  group isomorphisms ${\rm Gr}(B) \simeq \Lambda$ and ${\rm Gr}(B') \simeq \Lambda$.
\item The algebra map 
\[
\theta = (\pi \otimes \pi') \circ \Delta : \left\{ 
\begin{array}{ccc}
H & \longrightarrow & B \otimes B' \\
 x &\longmapsto & \pi(x_{(1)}) \otimes \pi'(x_{(2)})
\end{array} \right.
\]
is injective
\end{enumerate}
\end{definition}

It is shown in \cite[Theorem 8.1.1]{pw} that the Hopf algebras 
$\mathcal O({\rm GL}_q(n))$ and
 $\mathcal O({\rm SL}_q(n))$ have a bense big cell. More precisely
$(\mathcal O(B_q), \mathcal O(B'_q),\mathbb Z^n)$ is a dense big cell for 
$\mathcal O({\rm GL}_q(n))$, where $\mathcal O(B_q)$ and $\mathcal O(B'_q)$ are the 
natural $q$-analogues of the Borel subgroups of lower and upper triangular matrices.
This result was a crucial tool in the proof by Parshall and Wang that the simple
$\mathcal O({\rm GL}_q(n))$-comodules are  classified by the dominant
weights in $\mathbb Z^n$, using the Borel--Weil construction, similarly to the classical case
(see \cite[Theorems 8.3.1 and 8.7.2]{pw}). 

Note that our setting allows some ``degenerate'' examples, whose behavior might be quite different from the one of reductive algebraic groups. For instance, if $G$ is any connected solvable algebraic group, $T \subset G$ is a maximal torus, and $G' \subset G$ is a connected closed subgroup of $G$ containing $T$, then $(\mathcal{O}(G'),\mathcal{O}(G),X^*(T))$ is a dense big cell for $\mathcal{O}(G)$. Other examples include algebras of functions on Frobenius kernels of connected reductive algebraic groups in positive characteristic, see \cite[Lemma II.3.2]{jan}.

Our main result states that a Hopf algebra having a dense big cell
automatically has the Borel--Weil property.

\begin{theorem}\label{main}
Let $H$ be a Hopf algebra having a dense big cell.
Then $H$ has the Borel--Weil property. More precisely, if
 $(B,B',\Lambda)$ is a dense big cell for $H$, then $(H,B)$ is a Borel--Weil datum with weight group $\Lambda$.
\end{theorem}

A large part of the proof (given in the next section) consists in verifying that the arguments 
given in  \cite{pw} still hold in this more general framework.
At some occasions the arguments have to be modified
(and the statements are weaker), because Parshall and Wang
used an order on the abelian weight group, and here
we do not have a  natural
order in general on the possibly non commutative weight group $\Lambda$. 

The other consequence of
the absence of an order on the weight group is that we have no combinatorial
criterion that ensures that a weight is ``dominant''.
This is certainly something that is missing in our framework, but which seems difficult to reach at this level of generality. We do not have any
general statement in this direction, for which we probably would need some kind of non commutative
root system theory.


\section{Proof of Theorem \ref{main}}
\label{sec:proof}

First let us
fix some notation. Let $H$ be a Hopf algebra  having a dense big cell 
$(B,B',\Lambda)$. 
We let $\nu : \Lambda \longrightarrow {\rm Gr}(B)$ (respectively $\nu' : \Lambda \longrightarrow {\rm Gr}(B')$) be the inverse of the isomorphism induced by $\psi$ (respectively $\psi'$)
and for $\lambda \in \Lambda$, we denote by ${\rm Ind}_B^H(\lambda)$ the induced $H$-comodule ${\rm Ind}_B^H(\nu(\lambda))$. Similarly, we say that $\lambda \in \Lambda$ is a highest $B'$-weight if $\nu'(\lambda)$ is a highest $B'$-weight.

We begin by showing that a simple $H$-comodule has a unique $B'$-stable line, and hence
a highest $B'$-weight.

\begin{proposition}\label{highest}
 Let $V$ be a simple $H$-comodule. Then $V$ has a unique 
 $B'$-stable line, and hence a highest $B'$-weight. Thus we have a map
$${\rm Irr}(H) \longrightarrow {\rm Gr}(B')\simeq\Lambda$$
sending the isomorphism class of a simple $H$-comodule to its highest $B'$-weight.
\end{proposition}

\begin{proof}
The Hopf algebra $B'$ is pointed, hence $V$ has a $B'$-stable line. We have to prove that it is unique.
So assume that we have linearly independent vectors $v_1, v_2 \in V$
with
$$({\rm id}_V \otimes \pi')\circ \alpha(v_i) = v_i \otimes g_i, \ i=1,2$$
and $g_1,g_2 \in {\rm Gr}(B')$. Let us extend $v_1,v_2$ into a basis $v_1, \ldots , v_n$ of $V$,
and let $(x_{ij}) \in M_n(H)$
be the multiplicative matrix such that
$$\alpha(v_i) = \sum_j v_j \otimes x_{ji}.$$
We have 
\begin{equation}
\label{eqn:pi-pi'}
\pi'(x_{i1}) = \delta_{i1}g_1 \quad {\rm and} \quad
\pi'(x_{i2}) = \delta_{i2}g_2.
\end{equation}
By Lemma \ref{dependent}, the elements
$$\pi(x_{11}), \pi(x_{21}), \ldots , \pi(x_{n1}), \pi(x_{12}), \ldots , \pi(x_{n2})$$
are linearly dependent in the pointed Hopf algebra $B$, and hence there exists
some scalars $r_1, \ldots , r_n, s_1, \ldots s_n \in k$ with some $r_i$ or $s_i$ non zero such that
$$
\sum_i r_i\pi(x_{i1}) +s_i\pi(x_{i2})=0.$$ 
Hence we have
$$
 \sum_i \sum_l r_i\pi(x_{il})\otimes \pi(x_{l1}) +s_i\pi(x_{il})\otimes \pi(x_{l2}) =0$$
and then in $B \otimes k\Lambda$ we have
\begin{align*}
 0 & = \sum_i \sum_l r_i\pi(x_{il})\otimes \psi\pi(x_{l1}) +s_i\pi(x_{il})\otimes \psi\pi(x_{l2}) \\
& = \sum_i \sum_l r_i\pi(x_{il})\otimes \psi'\pi'(x_{l1}) +s_i\pi(x_{il})\otimes \psi'\pi'(x_{l2}) \\
& = \sum_i  r_i\pi(x_{i1})\otimes \psi'(g_1) +s_i\pi(x_{i2})\otimes \psi'(g_2)
\end{align*}
where in the last equality we have used \eqref{eqn:pi-pi'}.
Using the morphism $\nu'$,
we get in $B \otimes B'$
\begin{align*}
 0 &= \sum_i  r_i\pi(x_{i1})\otimes g_1 +s_i\pi(x_{i2})\otimes  g_2 \\
 &= \sum_i \sum_l r_i\pi(x_{il})\otimes \pi'(x_{l1}) +s_i\pi(x_{il})\otimes \pi'(x_{l2})\\
 & = \theta\left(\sum_i r_ix_{i1} +s_ix_{i2}\right) 
\end{align*}
and we conclude from the injectivity of $\theta$ that $\sum_i r_ix_{i1} +s_ix_{i2}=0$.
The comodule $V$ is simple, hence the elements $x_{ij}$ are linearly independent
in $H$, hence $r_1= \cdots =r_n=s_1= \cdots  =s_n=0$, a contradiction.

 It is easy to check that two isomorphic simple $H$-comodules  have the same  $B'$-weights, and this concludes the proof.
\end{proof}

We now proceed to study comodules induced by elements of $\Lambda \simeq {\rm Gr}(B)$.
The next result is similar to \cite[Theorem 8.3.1]{pw}, with the same proof,
that we include for the sake of completeness.

\begin{proposition}\label{pw}
 Let $\lambda \in \Lambda$. Assume that ${\rm Ind}_B^H(\lambda) \not =\{0\}$.
 Then ${\rm Ind}_B^H(\lambda)$ contains a unique $B'$-stable line
and hence a unique simple $H$-subcomodule, whose highest $B'$-weight is $\lambda$, and that we denote $L(\lambda)$. 
\end{proposition}

\begin{proof}
We already know that the $H$-comodule ${\rm Ind}_B^H(\lambda)$
contains a $B'$-stable line, and we have to prove that it is unique.
Let $x \in {\rm Ind}_B^H(\lambda)$, $x\not=0$, be such that
$k x$ is a $B'$-stable line. Then we have
$$\pi(x_{(1)}) \otimes x_{(2)}= \nu(\lambda) \otimes x \quad {\rm and} \quad
x_{(1)}\otimes\pi'(x_{(2)}) = x \otimes \nu'(\mu)$$
for some $\mu \in \Lambda$.
Hence we have
$$\pi(x)=\varepsilon(x)\nu(\lambda) \quad {\rm and} \quad
\pi'(x)= \varepsilon(x)\nu'(\mu)$$
and
$$\varepsilon(x)\lambda=\psi\circ \pi(x)=\psi'\circ\pi'(x) = \varepsilon(x) \mu.$$
On the other hand
$$\theta(x) = \pi(x_{(1)}) \otimes \pi'(x_{(2)})=\varepsilon(x)\nu(\lambda) \otimes \nu'(\mu).
$$
Hence we have $\varepsilon(x)\not=0$ since $\theta$ is injective, and we get
$\lambda=\mu$. Also the identity $\theta(x)=\varepsilon(x)\nu(\lambda)\otimes \nu'(\lambda)$ 
and the injectivity of $\theta$ show that ${\rm Ind}_B^H(\lambda)$ has at most one 
$B'$-stable line, and the highest weight is $\lambda$. Finally two simple subcomodules of ${\rm Ind}_B^H(\lambda)$ have a $B'$-stable line, which is the unique $B'$-stable line
of ${\rm Ind}_B^H(\lambda)$, hence they coincide. Thus ${\rm Ind}_B^H(\lambda)$
contains a unique simple $H$-comodule whose highest $B'$-weight is necessarily $\lambda$.
\end{proof}

\begin{rem}
It follows from Proposition \ref{pw} that the $H$-comodule ${\rm Ind}_B^H(\lambda)$ is indecomposable.
In particular, if $H$ is cosemisimple, then ${\rm Ind}_B^H(\lambda)$ is a simple $H$-comodule.
\end{rem}

\begin{proposition}\label{dominant}
 Let $V$ be a simple $H$-comodule with highest $B'$-weight $\lambda \in \Lambda$.
 Then $V$ is isomorphic to a subcomodule of ${\rm Ind}_B^H(\lambda)$, and in particular
 ${\rm Ind}_B^H(\lambda) \not =\{0\}$.
\end{proposition}

\begin{proof}
Since $B$ is pointed, there exists a basis $v_1, \ldots , v_n$ of $V$ such that
$$\alpha(v_i)=\sum_j v_j \otimes x_{ji}$$
and that the matrix $(\pi(x_{ij}))\in M_n(B)$ is lower triangular.
In particular $\pi(x_{1i}) =\delta_{1i}\nu(\mu)$ for some 
$\mu \in \Lambda$. Thus
$$(\pi \otimes {\rm id}_H)\circ \Delta(x_{1i})=\nu(\mu)\otimes x_{1i}$$
and hence $x_{1i} \in {\rm Ind}_B^H(\mu)$. The linear map
\[
\left\{
\begin{array}{ccc}
 V &\longrightarrow & H \\
v_i &\longmapsto & x_{1i}
\end{array}
\right.
\]
is $H$-colinear and injective since $V$ is simple, and its image
is contained in ${\rm Ind}_B^H(\mu)$. The previous proposition ensures that
$\lambda=\mu$, and we are done.
\end{proof}

We have now all the ingredients to prove Theorem \ref{main}, in the following form.

\begin{theorem}\label{remain}
 Let $H$ be a Hopf algebra having a dense big cell
 $(B,B',\Lambda)$, and let $\Lambda_+= \{\lambda \in \Lambda \ | \ {\rm Ind}_B^H(\lambda) \not = \{0\}\}$. Then the map
\[
\left\{
\begin{array}{ccc}
\Lambda_+ &\longrightarrow & {\rm Irr}(H) \\
\lambda &\longmapsto & [L(\lambda)]
\end{array}
\right.
\]
is a bijection, and $(H,B)$ is a Borel--Weil datum.
\end{theorem}

\begin{proof}
The above map was constructed in Proposition \ref{pw}, and 
we have to construct an inverse. If $V$ is a simple $H$-comodule, its
highest $B'$-weight $\lambda_V\in \Lambda$ (Proposition \ref{highest}) is in fact, by Proposition \ref{dominant},
 an element of $\Lambda_+$.
This gives  a map
 \[
 \left\{
 \begin{array}{ccc}
 {\rm Irr}(H) & \longrightarrow & \Lambda_{+} \\
 V & \longmapsto & \lambda_V
\end{array}
\right. .
\]
By definition the comodule $L(\lambda_V)$ is the unique simple comodule
of ${\rm Ind}_B^H(\lambda_V)$. Proposition \ref{dominant} ensures that
$V$ is isomorphic to a simple subcomodule of ${\rm Ind}_B^H(\lambda_V)$, and thus 
$L(\lambda_V) \simeq V$.

Starting with $\lambda \in \Lambda_+$, Proposition \ref{pw} ensures 
that $\lambda$ is the highest $B'$-weight of $L(\lambda)$, and the bijectivity of the map in the statement of the theorem is thus established. 

It remains to check that
$\Lambda_+$ is a submonoid of $\Lambda$. Let $\lambda, \mu \in \Lambda_+$ and let
$x \in {\rm Ind}_B^H(\lambda)$ and $y \in {\rm Ind}_B^H(\mu)$ be such that
$kx$ and $ky$ are $B'$-stable lines. Then by the proof of Proposition \ref{pw}
we have $\varepsilon(x) \not =0$ and $\varepsilon(y) \not=0$. This shows
that $xy\not =0$ and the argument in Remark \ref{submonoid}
works to conclude that $\lambda\mu \in \Lambda_+$.
\end{proof}

Let us note the following characterization of dominant weights, which will be useful later.

\begin{proposition}
\label{prop:dominant-weights}
Let $\lambda \in \Lambda$. Then $\lambda \in \Lambda_+$ if and only if $\nu(\lambda) \otimes \nu'(\lambda) \in \mathrm{Im}(\theta)$. 
\end{proposition}

\begin{proof}
The ``only if'' part was proved in the course of the proof of Proposition \ref{pw}. Now assume that there exists $x \in H$ such that $\theta(x)=\nu(\lambda) \otimes \nu'(\lambda)$. Then we claim that $x \in \mathrm{Ind}_B^H(\lambda)$, which will finish the proof. Indeed, we have to check that
\[
\pi(x_{(1)}) \otimes x_{(2)} = \nu(\lambda) \otimes x
\]
in $B \otimes H$. By injectivity of $\mathrm{id}_B \otimes \theta$, it is sufficient to check that
\[
\pi(x_{(1)}) \otimes \pi(x_{(2)}) \otimes \pi'(x_{(3)}) = \nu(\lambda) \otimes \pi(x_{(1)}) \otimes \pi'(x_{(2)}).
\]
However, by our assumption on $x$ both of these elements are equal to $\nu(\lambda) \otimes \nu(\lambda) \otimes \nu'(\lambda)$.
\end{proof}

\begin{rem}
\label{rem:new-big-cell}
Let $H$ be a Hopf algebra having a dense big cell $(B,B',\Lambda)$. Assume we are given some quotient Hopf algebras $C$ and $C'$ of $H$ such that
\begin{itemize} 
\item
$\pi$ (respectively $\pi'$) factors through the quotient $H \to C$ (respectively $H \to C'$);
\item
the induced morphism $C \to B$ (respectively $C' \to B'$) restricts to an isomorphism $\mathrm{Gr}(C) \xrightarrow{\sim} \mathrm{Gr}(B)$ (respectively $\mathrm{Gr}(C') \xrightarrow{\sim} \mathrm{Gr}(B')$).
\end{itemize} 
Then $(C,C',\Lambda)$ is also a dense big cell for $H$. Moreover, the dominant weights for both big cells coincide: indeed for any $\lambda \in \Lambda$ we have an inclusion $\mathrm{Ind}_{C}^G(\lambda) \subset \mathrm{Ind}_{B}^G(\lambda)$. Hence a dominant weight for $(C,C',\Lambda)$ is also dominant for $(B,B',\Lambda)$, and the associated simple $H$-comodules are the same. Now by Theorem \ref{remain} both sets parametrize isomorphism classes of simple $H$-comodules, hence they must coincide.
\end{rem}

\begin{rem}
 It is natural to wonder whether the spaces ${\rm Ind}_{B}^{H}(\lambda)$ always are finite-dimensional, as they are in the case of reductive groups. In fact this is not true at our level of generality: if $H$ is a Hopf algebra having a dense big cell $(B,B',\Lambda)$ and $U$ is a co-unipotent Hopf algebra (a pointed Hopf algebra with $\mathrm{Gr}(U)=\{1\}$), then $H \otimes U$ has a dense big cell $(B,B' \otimes U,\Lambda)$, the morphism $H \otimes U \to B$ being given by $\psi \otimes \varepsilon_U$.
We have
$U  \subset {\rm Ind}_{B}^{H \otimes U}(1)$, so the latter space is infinite-dimensional if $U$ is.
\end{rem}

\section{First example: $\mathcal O({\rm SL}_q(2))$}
\label{sec:SL2}

\subsection{Existence of a dense big cell}

As an example, let us use prove that the Hopf algebra $\mathcal O({\rm SL}_q(2))$ has a dense big cell (for $q \in k^\times$) with weights $\mathbb{Z}$ and dominant weights $\mathbb{N}$. Of course this is well known at least since \cite{pw}, but the particular proof will be useful later. The proof uses the following easy lemma.

\begin{lemma}\label{grading2}
 Let $V$, $W$ be vector spaces graded by a totally ordered set $S$:
$$V = \bigoplus_{s \in S}V^s, \quad W= \bigoplus_{s\in S} W^s.$$
For $s \in S$, denote by $p_s$ (resp. $q_s$) the canonical projection of $V$ (resp. $W$) on $V^s$ (resp. $W^s$).
Let $\theta : V \rightarrow W$ be a linear map. Assume that for any $s \in S$, we have $\theta(V^s) \subset W^{\leq s}$ (where $W^{\leq s}=\bigoplus_{t \leq s} W^t$) and that $q_s \circ \theta_{|V^s}$ is injective. Then $\theta$ is injective. If moreover $S$ has a smallest element $0$, then for any $v \in V$, if $\theta(v) \in W^0$, then $v \in V^0$.
\end{lemma}

\begin{proof}
The first claim follows from the observation that a filtered morphism whose associated graded is injective is itself injective. The second claim can be proved similarly.
\end{proof}

Recall that $\mathcal O({\rm SL}_q(2))$ is the algebra presented by generators
$a,b,c,d$ and relations
$$ba = qab \ ; \ ca = qac \ ; \ db = qbd \ ; \ 
dc = qcd \ ; \ cb = bc = q(ad-1) \ ; da = qbc +1.$$
The coalgebra structure is determined by the condition that the matrix $\begin{pmatrix} a & b \\ c & d\end{pmatrix}$
is multiplicative. The Hopf algebra $\mathcal O(B_q)$ (resp. $\mathcal O(B_q')$) is the quotient of  $\mathcal O({\rm SL}_q(2))$
by the relation $b=0$ (resp. $c=0$). We denote by $\pi : \mathcal O({\rm SL}_q(2)) \rightarrow \mathcal O(B_q)$, $\pi' : \mathcal O({\rm SL}_q(2)) \rightarrow \mathcal O(B_q')$ the respective canonical projections and by
$\psi : \mathcal O(B_q) \rightarrow k[t,t^{-1}] \simeq k \mathbb Z$, $\psi' : \mathcal O(B_q') \rightarrow k[t,t^{-1}] \simeq k \mathbb Z$ the unique Hopf algebra maps such that $\psi(a)=\psi'(a)=t$, $\psi(c)=\psi'(b)=0$.

\begin{proposition}\label{bigsl2}
 The Hopf algebra $\mathcal O({\rm SL}_q(2))$ has $(\mathcal O(B_q), \mathcal O(B'_q), \mathbb Z)$ as a dense big cell, with  $\mathbb{N}$ being the set of dominant weights.
\end{proposition}

%

\begin{proof}
Recall (see e.g.~\cite{ks}) that
$\mathcal O({\rm SL}_q(2))$
has a basis given by the elements of the form
\begin{equation}
\label{eqn:basis-SL2}
a^i b^j c^k, \ i,j,k \geq 0 \quad \text{and} \quad b^j c^k d^l, \ j,k \geq 0, \ l >0.
\end{equation}
For $n \geq 0$, we denote by $\mathcal O({\rm SL}_q(2))^n \subset \mathcal O({\rm SL}_q(2))$ the subspace generated by elements $a^i b^j c^k$ with $n=k+j$ and $b^j c^k d^l$ with $n=j+k+2l$. (We will say that such elements have ``degree'' $n$.) Then we have
\[
\mathcal O({\rm SL}_q(2)) = \bigoplus_{n \geq 0} \, \mathcal O({\rm SL}_q(2))^n.
\]
Similarly, $\mathcal{O}(B_q)$, respectively $\mathcal{O}(B_q')$, has a basis given by elements of the form
\begin{equation}
\label{eqn:basis-Bq}
a^i c^j, \ i \in \mathbb{Z}, \ j \geq 0, \quad \text{respectively} \quad a^k b^l, \ k \in \mathbb{Z}, \ l \geq 0,
\end{equation}
and for $n \geq 0$ we denote by $\bigl( \mathcal{O}(B_q) \otimes \mathcal{O}(B'_q) \bigr)^n$ the subspace generated by elements $a^i c^j \otimes a^{k} b^{l}$ with $n=j+l$. As above we have
\[
\mathcal{O}(B_q) \otimes \mathcal{O}(B'_q) = \bigoplus_{n \geq 0} \, \bigl( \mathcal{O}(B_q) \otimes \mathcal{O}(B'_q) \bigr)^n.
\]
It is a direct verification to check that
\[
\theta(\mathcal O({\rm SL}_q(2))^{n}) \subset \bigl( \mathcal{O}(B_q) \otimes \mathcal{O}(B'_q) \bigr)^{\leq n}
\]
for all $n \geq 0$. It is also not difficult to see that, in the notation of Lemma \ref{grading2}, for any $n$ the map $q_n\theta_{|\mathcal O({\rm SL}_q(2))^n}$ is injective; in fact it sends the basis of $\mathcal O({\rm SL}_q(2))^n$ obtained from \eqref{eqn:basis-SL2} to a family of elements consisting of non zero scalar multiples of distinct vectors in the basis of $\bigl( \mathcal{O}(B_q) \otimes \mathcal{O}(B'_q) \bigr)^{n}$ obtained from \eqref{eqn:basis-Bq}. We conclude by Lemma \ref{grading2} that $\theta$ is injective. 

Now let $a^m$, $m \in \mathbb Z$, be a dominant weight. By Proposition \ref{prop:dominant-weights}, there exists $x \in  \mathcal O({\rm SL}_q(2))$ such that $\theta(x)= a^m \otimes a^m$. Then by Lemma \ref{grading2} we have $x \in \mathcal O({\rm SL}_q(2))^0=k[a]$, and it follows that $m \geq 0$. 
\end{proof}

\subsection{Grothendieck ring}

For later use, in this subsection we briefly recall how one can describe the Grothendieck ring $K \bigl( \mathcal{O}(\mathrm{SL}_q(2)) \bigr)$ of $\mathcal{O}(\mathrm{SL}_q(2))$.

We have (re)proved in Proposition \ref{bigsl2} that $\mathcal{O}(\mathrm{SL}_q(2))$ has a dense big cell with dominant weights $\mathbb{N}$. Hence by Theorem \ref{remain} its simple comodules are parametrized by $\mathbb{N}$: we denote by $L(i)$ the simple comodule associated with $i$. Using e.g.~formal characters (see \cite[Theorem 8.2.1 and Proposition 8.8.1]{pw}), one can check the well-known fact that, in $K \bigl( \mathcal{O}(\mathrm{SL}_q(2)) \bigr)$, for any  $i,j \geq 0$ we have
\begin{equation}
\label{eqn:tensor-product-sl2}
[L(i)] \cdot [L(j)] = [L(i+j)] + \sum_{k=0}^{i+j-1} a_{ij}^k \cdot [L(k)]
\end{equation}
for some $a_{ij}^k \geq 0$.
Using an easy filtration argument we deduce the following.

\begin{lemma}
\label{lem:Grothendieck-SL2}
The ring morphism
\[
\mathbb{Z}[T] \to K \bigl( \mathcal{O}(\mathrm{SL}_q(2)) \bigr)
\]
sending $T$ to $[L(1)]$ is an isomorphism.
\end{lemma}





 \section{Free products} 
We now study free products of Hopf algebras. 
We have not been able to show that the free product of Hopf algebras having a dense big cell still has a dense big cell, and 
we have just a very particular result in this direction (to be used later in the paper). This should not cause too much trouble since the simple comodules of a free product can be easily described in terms of the simple comodules of the factors.


\subsection{Simple comodules of free products of Hopf algebras}
 
 We begin by recalling some basic vocabulary and facts. Let $A,B$ be algebras, and let $A*B$ be their free product (coproduct in the category of unital algebras).
 For future convenience, we recall one possible construction of $A*B$ (see \cite{ns}). First, let us say that a subspace $X$ of an algebra $A$ is an \textbf{augmentation subspace of $A$} if $A=k1 \oplus X$. Now if $X=Z_1$ and $Y=Z_2$ are  augmentation subspaces of $A$ and $B$ respectively, we have $$A*B= k1 \oplus \left( \bigoplus_{m=1}^\infty \bigoplus_{i_1 \not = i_2 \not = \cdots \not = i_m} Z_{i_1} \otimes \cdots \otimes Z_{i_m}\right).$$
The right-handed term is denoted $X*Y$; this is an augmentation subspace of $A*B$. If $\{a_{i}, i \in I\}$, $\{b_j, \ j \in J\}$ denote respective bases of $X$ and $Y$, then the elements
\begin{equation}
\label{eqn:basis-free-product}
\begin{split}
a_{i_1} b_{j_1} \cdots a_{i_m}b_{j_m}a_{i_{m+1}}, & \ i_1, \ldots , i_{m+1} \in I, \  j_1, \ldots ,j_m \in J, \ m \geq 0 \\
b_{j_1} a_{i_1}  \cdots b_{j_m} a_{i_m}b_{j_{m+1}}, & \ i_1, \ldots , i_{m} \in I, \  j_1, \ldots , \ j_{m+1} \in J, m \geq 0\\
a_{i_1} b_{j_1} \cdots a_{i_m} b_{j_m}, & \ i_1, \ldots , i_m \in I, \  j_1, \ldots , j_m \in J, \ m \geq 1\\
b_{j_1} a_{i_1}\cdots b_{j_m}a_{i_m}, & \ i_1, \ldots , i_m \in I, \  j_1, \ldots , j_m \in J, \ m \geq 1
\end{split}
\end{equation}
form a basis of $X*Y$.

Let $H, L$ be Hopf algebras. Recall \cite{wa1} that the free product algebra $H*L$ has a unique Hopf algebra structure such that the canonical morphisms $H \longrightarrow H*L$ and $L \longrightarrow H*L$ are Hopf algebra maps. An $H*L$-comodule is said to be a simple alternated $H*L$-comodule if it has the form $V_1 \otimes \cdots \otimes V_n$,where each $V_i$ is a simple non-trivial $H$-comodule or $L$-comodule, and if $V_i$ is an $H$-comodule, then $V_{i+1}$ is an $L$-comodule, and conversely. It is proved in \cite[Theorem 3.10]{wa1} that if $H$ and $L$ are cosemisimple, then the simple $H*L$-comodules are exactly the simple alternated comodules.
We first note that the result generalizes as follows. We thank the referee for explaining us how to remove some technical assumption on $H$ and $L$ in the statement.

\begin{proposition}\label{simplefree}
 The simple comodules over the free product Hopf algebra $H*L$ are the simple alternated $H*L$-comodules. In particular, the natural ring morphism
 \[
 K(H) * K(L) \to K(H * L)
 \]
 between Grothendieck rings
 is an isomorphism.
\end{proposition}

\begin{proof}
To prove that the simple alternated comodules are indeed simple comodules, it suffices to prove that their coefficients are linearly independent. This easily follows from the description of bases of the free product at the beginning of the section, using the fact that one can choose the augmentation subspaces of $H$ and $L$ so that they contain the coefficients of all non-trivial simple comodules.
 Conversely, we will show that the family formed by the simple alternated comodules satisfies the assumptions in Proposition \ref{completeset}, which implies the expected result. The first assumption in Proposition \ref{completeset} is clearly satisfied, and to check the second one, we just have to show that 
$(H*L)_0$ is contained in the subalgebra generated by $H_0$ and $L_0$. Since $H*L$ is generated as an algebra by $H$ and $L$, multiplication induces a coalgebra surjection
$$\bigoplus_{k\geq 0} \bigoplus_{n_i,m_i\geq 0} H^{\otimes n_1} \otimes L^{\otimes m_1} \otimes \cdots \otimes H^{n_k} \otimes L^{m_k} \rightarrow H*L$$ 
We conclude by combining the following two facts: (a) for a coalgebra surjection $f :C \rightarrow D$ we have $D_0\subset f(C_0)$ by \cite[Corollary 5.3.5]{mo}, and (b) the coradical of  a tensor product of coalgebras is contained in the tensor product of coradicals by \cite[Lemma 5.1.10]{mo}. 
  The last assertion is an obvious consequence of the first one.
\end{proof}

\begin{corollary}
  Let $H, L$ be Hopf algebras having the Borel--Weil property, with respective weight groups $\Lambda$ and $\Gamma$. Then there is a bijection between the set of isomorphism classes of simple comodules over  $H*L$ and the submonoid of  $\Lambda*\Gamma$ generated by $\Lambda_+$ and $\Gamma_+$. 
\end{corollary}

\begin{proof}
 There is an obvious bijection between the set of isomorphism classes of simple alternated comodules and the submonoid of $\Lambda*\Gamma$ generated by $\Lambda_+$ and $\Gamma_+$, so the result follows from Proposition \ref{simplefree}.
\end{proof}

\subsection{The case of $\mathcal O({\rm SL}_q(2))*k\Gamma$}

The notations in this subsection and the next one are those introduced in Section \ref{sec:SL2}.

\begin{proposition}\label{freesl2}
 Let $\Gamma$ be a discrete group. Then the Hopf algebra $\mathcal O({\rm SL}_q(2))*k\Gamma$ has $$(\mathcal O(B_q)*k\Gamma, \mathcal O(B'_q)*k\Gamma, \mathbb Z*\Gamma)$$ as a dense big cell. Moreover $(\mathbb Z*\Gamma)_+$
is the submonoid generated by $t$ and $\Gamma$, and every simple $\mathcal O({\rm SL}_q(2))*k\Gamma$-comodule is a simple alternated comodule.
\end{proposition}

\begin{proof}
 The structural Hopf algebra maps are
 \begin{gather*}
\pi*{\rm id} : \mathcal O({\rm SL}_q(2))*k\Gamma \rightarrow \mathcal O(B_q)*k\Gamma, \ 
\pi'*{\rm id} : \mathcal O({\rm SL}_q(2))*k\Gamma \rightarrow \mathcal O(B_q')*k\Gamma\\
\psi*{\rm id} : \mathcal O(B_q)*k\Gamma \rightarrow k[t,t^{-1}]*k\Gamma \simeq k \mathbb Z*\Gamma, \ \psi'*{\rm id} : \mathcal O(B_q')*k\Gamma \rightarrow k[t,t^{-1}]*k\Gamma \simeq k \mathbb Z*\Gamma
\end{gather*}
where $\pi$, $\pi'$, $\psi$, $\psi'$ have been defined before Proposition \ref{bigsl2}, and $*$ is the free product of algebra maps. 

We have to prove that the corresponding algebra map $\theta$ is injective. For this we use the same strategy as for Proposition \ref{bigsl2}. Starting with the basis 
\[
a^i b^j c^k, \ i,j,k \geq 0, \ i+j+k >0, \qquad b^j c^k d^l, \ j,k \geq 0, \ l>0
\]
of an augmentation subspace of $\mathcal O({\rm SL}_q(2))$ and the basis $\Gamma \smallsetminus \{1\}$ of an augmentation subspace of $k\Gamma$ we obtain a basis of $\mathcal O({\rm SL}_q(2))*k\Gamma$ as in \eqref{eqn:basis-free-product}. If $n>0$, we denote by $\bigl( \mathcal O({\rm SL}_q(2))*k\Gamma \bigr)^n$ the subspace generated by such elements where the sum of the ``degrees'' of the elements in $\mathcal O({\rm SL}_q(2))$ (in the sense of the proof of Proposition \ref{bigsl2}) is $n$. We also define $\bigl( \mathcal O({\rm SL}_q(2))*k\Gamma \bigr)^0$ as the subspace generated by such elements where all elements of $\mathcal O({\rm SL}_q(2))$ have ``degree'' $0$. (Note that this includes the unit of $\mathcal O({\rm SL}_q(2))*k\Gamma$.)

One can make similar definitions for $\bigl( \mathcal{O}(B_q)*k\Gamma \bigr) \otimes \bigl( \mathcal{O}(B_q')*k\Gamma \bigr)$ (using again the ``degrees'' considered in the proof of Proposition \ref{bigsl2}). Then it is easy to check (using the fact that the grading on $\bigl( \mathcal{O}(B_q)*k\Gamma \bigr) \otimes \bigl( \mathcal{O}(B_q')*k\Gamma \bigr)$ is an \emph{algebra} grading) that we have
\[
\theta \Bigl( \bigl( \mathcal O({\rm SL}_q(2))*k\Gamma \bigr)^n \Bigr) \subset \Bigl( \bigl( \mathcal{O}(B_q)*k\Gamma \bigr) \otimes \bigl( \mathcal{O}(B_q')*k\Gamma \bigr) \Bigr)^{\leq n}
\]
for any $n \geq 0$, and that (in the notation of Lemma \ref{grading2}) for any $n$ the map $q_n \theta_{|( \mathcal O({\rm SL}_q(2))*k\Gamma)^n}$ is injective. We deduce the injectivity of $\theta$ using Lemma \ref{grading2}.

By Proposition \ref{simplefree}, every simple $\mathcal O({\rm SL}_q(2))*k\Gamma$-comodule is a simple alternated comodule, whose highest weight clearly belongs to the submonoid generated by $a$ ($=\nu(t)$) and $\Gamma$.
Conversely, it is obvious that an element of this submonoid is the highest weight of a unique simple alternated comodule. 
\end{proof}


\subsection{Grothendieck ring}
\label{ss:Groth-ring-free-product}

Combining Lemma \ref{lem:Grothendieck-SL2} and Proposition \ref{simplefree} we obtain the following description of the Grothendieck ring $K \bigl( \mathcal{O}(\mathrm{SL}_q(2)) * k \Gamma \bigr)$ of the tensor category of finite dimensional $ \mathcal{O}(\mathrm{SL}_q(2)) * k \Gamma$-comodules.

\begin{lemma}
\label{lem:Grothendieck-free-product}
The ring morphism
\[
\varphi : \mathbb{Z}[T] * \mathbb{Z}\Gamma \to K \bigl( \mathcal{O}(\mathrm{SL}_q(2)) * k \Gamma \bigr)
\]
sending $T$ to $[L(1)]$ and $\gamma \in \Gamma$ to $[L(\gamma)]$ is an isomorphism.
\end{lemma}

\section{Hopf subalgebras}
In this short section, we discuss Hopf subalgebras.  The following preliminary result is due to the referee; it enabled us to remove a superfluous assumption in Proposition \ref{sub}.

\begin{lemma}
 Let $C \subset D$ be a subcoalgebra of a pointed coalgebra, and let $f : D \rightarrow D_0$ be a coalgebra map with $f_{|D_0}= {\rm id}_{D_0}$. Then we have $f(C) \subset C_0$. 
\end{lemma}

\begin{proof}
The coalgebra $C$ is also pointed, with $C_n=D_n\cap C$ for any $n \geq 0$, where $C_n$, $D_n$ denote the $n$th part 
of the respective coradical filtrations (see \cite[Lemma 5.2.12]{mo}). We show that $f(C_n) \subset C_0$ by induction on $n\geq 0$, the case $n=0$ being obvious.
Let $n\geq 1$ and assume that $f(C_{n-1}) \subset C_0$. Let $c \in C_n$. By the Taft--Wilson Theorem (see \cite[Theorem 5.4.1]{mo}), we can write
\[
c= \sum_{g,h \in {\rm Gr}(C)} c_{g,h} \quad \text{with} \quad \Delta(c_{g,h}) = c_{g,h}\otimes g + h \otimes c_{g,h} + w
\]
for some $w=w_{g,h} \in C_{n-1}\otimes C_{n-1}$.
For any $g,h \in {\rm Gr}(C)$, we have, using the induction assumption,  
\[
\Delta f(c_{g,h})=(f\otimes f)\Delta(c_{g,h})=
f(c_{g,h})\otimes g + h \otimes f(c_{g,h}) + (f \otimes f)(w)\in D_0 \otimes C_0 + C_0 \otimes D_0.
\]
Since any element of $D_0$ is a linear combination of group-like elements, this implies that $f(c_{g,h}) \in C_0$. It follows that $f(c) \in C_0$.
 \end{proof}

\begin{proposition}\label{sub}
 Let $H$ be a Hopf algebra having a dense big cell $(B,B',\Lambda)$ with corresponding Hopf algebra maps 
\[
\xymatrix{
H \ar[r]^-{\pi} & B \ar[r]^-{\psi} & k\Lambda
}
\quad \text{and} \quad
\xymatrix{
H \ar[r]^-{\pi'} & B' \ar[r]^-{\psi'} & k\Lambda.
}
\]
Let $\nu : \Lambda \longrightarrow {\rm Gr}(B)$ and 
$\nu' : \Lambda \longrightarrow {\rm Gr}(B')$
be as in Section \ref{sec:proof}. Let $A \subset H$ be a Hopf subalgebra. Then 
we have $\psi ({\rm Gr}(\pi(A)))=\psi'({\rm Gr}(\pi'(A))=: \Lambda_A$, and 
 $(\pi(A), \pi'(A), \Lambda_A)$ is a dense big cell for $A$. Moreover we have $(\Lambda_A)_+ \subset \Lambda_+ \cap \Lambda_A$, with equality if the following conditions are fulfilled:
\begin{enumerate}
 \item  $A\subset H$ is a normal Hopf subalgebra;
\item $H$ is faithfully flat as a left (or right) $A$-module;
\item  ${\rm Ker}(\pi) \subset A^+H$.
\end{enumerate}
\end{proposition}

\begin{proof}
The previous lemma, applied to the inclusion $\pi(A) \subset B$ and the morphism $f=\nu\psi$, ensures that
$\nu\psi\pi(A) \subset \pi(A)$, and hence that $\psi(\mathrm{Gr}(\pi(A))) = \Lambda \cap \psi\pi(A)$. Similarly we have $\psi'(\mathrm{Gr}(\pi'(A))) = \Lambda \cap \psi'\pi'(A)$. Hence $\psi ({\rm Gr}(\pi(A)))=\psi'({\rm Gr}(\pi'(A))=: \Lambda_A$.
Then it is easy to check that
 $(\pi(A), \pi'(A), \Lambda_A)$ is a dense big cell for $A$. It is also clear that $(\Lambda_A)_+ \subset \Lambda_+ \cap \Lambda_A$, since $\mathrm{Ind}_{\pi(A)}^A(\lambda) \subset \mathrm{Ind}_{B}^H(\lambda)$ for any $\lambda \in \Lambda_A$.
 
 Now we assume that conditions $(1)$, $(2)$, $(3)$ are fulfilled, and let $\lambda \in \Lambda_+ \cap \Lambda_A$. Then ${\rm Ind}_B^H(\lambda)\not=\{0\}$ and $\lambda=\psi\pi(a)$ for some $a$ in $A$ with $\pi(a)$ group-like. Hence there exists $ x \in H$, $x \not=0$, with 
$$\pi(x_{(1)})\otimes x_{(2)} = \nu(\lambda)  \otimes x=\nu\psi\pi(a) \otimes x= \pi(a) \otimes x.$$
 By the first assumption we can form the quotient Hopf algebra $H/A^+H$, and we denote by $p : H \rightarrow H/A^+H$ the quotient map. By the third assumption there exists a unique Hopf algebra map $\overline{p} : B \rightarrow H/A^+H$ such that $\overline{p} \circ \pi =p$. We have
$\overline{p}\pi(x_{(1)})\otimes x_{(2)} = \overline{p} \pi(a) \otimes x$, hence 
\begin{equation}
\label{eqn:sub}
p(x_{(1)}) \otimes x_{(2)} = p(a)\otimes x.
\end{equation}
By definition $p(a)=\varepsilon(a) = \varepsilon\pi(a)=1$, hence $x \in {^{{\rm co}p}H}$.
But assumptions (1) and (2) yield that $A=H^{{\rm co}p}={^{{\rm co}p}H}$ (\cite{sch}), so \eqref{eqn:sub} implies that $x \in A$. Finally we obtain that $x \in {\rm Ind}_{\pi(A)}^A(\lambda)\not=\{0\}$, hence $\lambda \in (\Lambda_A)_+$.
\end{proof}

The proposition can be used, for example, to show that $\mathcal O({\rm PGL}_q(n))$ has a dense big cell, and to get its dominant weights from those of $\mathcal O({\rm GL}_q(n))$.

\begin{rem}
Consider the setting of Proposition \ref{sub}, and let $\lambda \in (\Lambda_A)_+$. We have seen that there is a natural inclusion of $H$-comodules $\mathrm{Ind}_{\pi(A)}^A(\lambda) \subset \mathrm{Ind}_{B}^H(\lambda)$. In particular, the simple $H$-comodule associated with $\lambda$ is the restriction of the simple $A$-comodule associated with $\lambda$.
\end{rem}

\section{Example: universal cosovereign Hopf algebras}

In this section we study the example of universal cosovereign Hopf algebras.
This will provide a natural example of a Hopf algebra having a dense big cell
with weight group $\mathbb F_2$ (the free group on two generators) and 
hence a Borel--Weil datum with weight group $\mathbb F_2$.

\subsection{Existence of a dense big cell}
\label{ss:cosovereign-big-cell}

Let $F \in {\rm GL}(n,k)$.
Following \cite{bi1} we consider the algebra $H(F)$
generated by
 $(u_{ij})_{1 \leq i,j \leq n}$ and
 $(v_{ij})_{1 \leq i,j \leq n}$, with relations:
$$ {u} {v^t} = { v^t} u = I_n ; \quad {vF} {u^t} F^{-1} = 
{F} {u^t} F^{-1}v = I_n,$$
where $u= (u_{ij})$, $v = (v_{ij})$ and $I_n$ is
the identity $n \times n$ matrix. The algebra
$H(F)$ has a  Hopf algebra structure
defined by
\begin{gather*}
\Delta(u_{ij}) = \sum_k u_{ik} \otimes u_{kj}, \quad
\Delta(v_{ij}) = \sum_k v_{ik} \otimes v_{kj}, \\
\varepsilon (u_{ij}) = \varepsilon (v_{ij}) = \delta_{ij}, \quad 
S(u) = {v^t}, \quad S(v) = F { u^t} F^{-1}.
\end{gather*}

The universal property of the Hopf algebras $H(F)$ \cite{bi1} shows that they
play, in the category of Hopf algebras,
a role that is similar to the one of $\mathcal O({\rm GL}(n,k))$  
in the category of commutative Hopf algebras: in particular any 
finitely generated Hopf algebra having all its finite-dimensional comodules
isomorphic to their bidual is a quotient of $H(F)$ for some $F$.
Hence one might say that they correspond to ``universal'' quantum groups.

When $k=\mathbb C$ and $F$ is a positive matrix, the Hopf algebra 
$H(F)$ corresponds to  Van Daele and Wang's universal compact quantum groups
\cite{vdw}. In this case the simple comodules have been classified by Banica \cite{ba}; they are
naturally labelled by the free monoid on two generators $\mathbb N * \mathbb N$.
More generally, if $k$ has characteristic zero, the matrices $F$ for which
$H(F)$ is cosemisimple have been determined in \cite{bi2}, and it is shown there 
that the classification of simple comodules given in \cite{ba} remains valid.

We use the techniques of the previous sections 
to classify the simple $H(F)$-comodules when 
${\rm tr}(F) \not= 0$ and $ {\rm tr} (F^{-1})\not=0$ or ${\rm tr}(F)=0={\rm tr} (F^{-1})$, removing
any assumption on $k$ and the genericity assumption in \cite{bi2}.
We show that the simple $H(F)$-comodules are still  labelled by the free monoid on two generators $\mathbb N * \mathbb N$. This was previously shown by Chirvasitu \cite{chi}, but the techniques used in the present paper provide explicit models for the simple comodules.

So let $F \in {\rm GL}(n,k)$ with $n \geq 2$,
${\rm tr}(F) \not= 0$ and  $ {\rm tr} (F^{-1})\not=0$ or ${\rm tr}(F)=0={\rm tr} (F^{-1})$. Such a matrix is said to be normalizable.
Let $q \in k^\times$ be such that
$q^2 -\sqrt{{\rm tr}(F){\rm tr}(F^{-1})}q +1 = 0$. Put 
$$H(q) = H\left( \Bigl( \begin{smallmatrix} q^{-1} & 0 \\
                          0 & q \\
       \end{smallmatrix} \Bigr) \right).$$
It is shown in \cite{bi2} that
the tensor categories of $H(F)$-comodules and $H(q)$-comodules are equivalent.
Hence we concentrate on the Hopf algebras $H(q)$.

We denote by $B(q)$ (resp. $B'(q)$) the quotient of $H(q)$ by the relations $u_{12}=0=v_{21}$ (resp. $u_{21}=0=v_{12}$) and by 
$$\pi : H(q) \rightarrow B(q), \quad \pi' : H(q) \rightarrow  B'(q)$$ the respective canonical projections. The algebras $B(q)$ and $B'(q)$ have a unique Hopf algebra structure such that $\pi$ and $\pi'$ are Hopf algebra maps. Moreover $B(q)$ and $B'(q)$ are pointed by Proposition \ref{pointed}.   
Let $\mathbb F_2$ be the free group on two generators $u_1$ and $u_2$. There are (surjective) Hopf algebra morphisms 
$$\psi : B(q) \rightarrow k\mathbb F_2, \quad \psi' : B'(q) \rightarrow k\mathbb F_2$$
such that $\psi(u_{ij})=\delta_{ij}u_i=\psi'(u_{ij})$ and $\psi(v_{ij})=\delta_{ij}u_{i}^{-1}=\psi'(v_{ij})$.
It is easy to see that $\psi$ and $\psi'$ induce isomorphisms ${\rm Gr}(B(q)) \simeq \mathbb F_2 \simeq {\rm Gr}(B'(q))$, see Proposition \ref{pointed}.

\begin{theorem}\label{bigfree}
 The Hopf algebra $H(q)$ has $(B(q), B'(q), \mathbb F_2)$ as a dense big cell, and $(\mathbb F_2)_+$ is the submonoid of $\mathbb F_2$ generated by $\alpha = u_1$ and $\beta=u_2^{-1}$. In particular there exists an explicit bijection ${\rm Irr}(H(q)) \simeq \mathbb N * \mathbb N$.
\end{theorem}

\begin{proof}
It is shown in \cite{bi2}, Section 3, that there exists a Hopf algebra embedding
\[
\iota : H(q) \hookrightarrow \mathcal O({\rm SL}_q(2))*k[z,z^{-1}]
\]
such that
\[
\iota
\begin{pmatrix}
 u_{11} & u_{12} \\
u_{21} & u_{22}
\end{pmatrix} = \begin{pmatrix}
 za & zb \\
zc & zd
\end{pmatrix}, \quad
\iota \begin{pmatrix}
 v_{11} & v_{12} \\
v_{21} & v_{22}
\end{pmatrix} =
\begin{pmatrix}
 dz^{-1} & -q^{-1}cz^{-1} \\
-qbz^{-1} & az^{-1}
\end{pmatrix}.
\]
Hence one can deduce from Propositions \ref{freesl2} and \ref{sub} that the Hopf algebra $H(q)$ has a dense big cell $(\pi(H(q)), \pi'(H(q)), \Lambda)$ where $\pi,\pi'$ are the morphisms which define the dense big cell of $\mathcal O({\rm SL}_q(2))*k[z,z^{-1}]$ as in Proposition \ref{freesl2}, and $\Lambda$ is the subgroup of $\mathbb{Z} * \mathbb{Z}$ (with generators $t$ and $z$, where $t$ is as in Section \ref{sec:SL2}) generated by $zt$ and $zt^{-1}$. Now we observe that the composition $\psi \circ \iota$, respectively $\psi' \circ \iota$, factors through the morphism
\[
\eta :B(q) \longrightarrow \mathcal O(B_q)*k[z,z^{-1}], \quad \text{respectively} \quad \eta' :B'(q) \longrightarrow \mathcal O(B'_q)*k[z,z^{-1}]
\]
defined by
\[
\eta \begin{pmatrix}
u_{11} & 0 \\
u_{21} & u_{22}
\end{pmatrix} = \begin{pmatrix}
za & 0 \\
zc & za^{-1}
\end{pmatrix}, \quad \text{respectively} \quad
\eta' \begin{pmatrix}
u_{11} & u_{12} \\
0 & u_{22}
\end{pmatrix} = \begin{pmatrix}
za & zb \\
0 & za^{-1}
\end{pmatrix}.
\]
Hence, by Remark \ref{rem:new-big-cell}, $(B(q), B'(q), \mathbb F_2)$ is indeed a dense big cell for $H(q)$. 

It remains to determine the dominant weights. We observe that $u_{11} \in {\rm Ind}_{B(q)}^{H(q)}(\alpha)$ and $v_{22} \in {\rm Ind}_{B(q)}^{H(q)}(\beta)$. Hence, as $(\mathbb F_2)_+$ is a submonoid of $\mathbb{F}_2$ by definition of a Borel--Weil datum, it contains the monoid generated by $\alpha$ and $\beta$.
(In fact it is easy to construct directly a non-zero element in ${\rm Ind}_{B(q)}^{H(q)}(\lambda)$ for any $\lambda$ in this monoid.)
To prove the reverse inclusion, note that the Hopf algebra embedding  $\iota$ induces an embedding of weight groups $\iota' : \mathbb F_2 \hookrightarrow \mathbb Z * \mathbb Z\simeq \mathbb F_2$ such that $\iota'(u_1) = zt$ and $\iota'(u_2) = zt^{-1}$. If $\lambda \in \mathbb{F}_2$ is dominant for $H(q)$, by Proposition \ref{sub} $\iota'(\lambda) \in (\mathbb Z * \mathbb Z)_+$, and by Proposition \ref{freesl2}, such an element is a word in $z,z^{-1}, t$. It is not difficult to see that a word in $u_1$, $u_1^{-1}$, $u_2$, $u_2^{-1}$ which is  a word in $z,z^{-1}, t$ is in fact a word in $\alpha=u_1$ and $\beta=u_2^{-1}$. We conclude that the elements of $(\mathbb F_2)_+$ are the words in $\alpha$, $\beta$, as required.
\end{proof}

\subsection{Simple comodules}

Using Theorem \ref{bigfree}, we are now able to provide explicit models for the simple comodules over the Hopf algebras $H(F)$, with $F$ normalizable. For this, we introduce $R(F)$, the algebra generated by
$x_1, \ldots, x_n$ and $y_1, \ldots , y_n$, with relations
$$\sum_{k=1}^n x_ky_k=0, \ \sum_{k,l=1}^nF_{kl}y_kx_l=F_{n1}$$ 
where $F=(F_{ij})$. The algebra $R(F)$ is naturally graded by the free monoid $\mathbb N * \mathbb N = \langle \alpha , \beta\rangle$, with $\deg(x_i)=\alpha$ and $\deg(y_i)=\beta$:
$$R(F)= \bigoplus_{\lambda \in  \mathbb N * \mathbb N}R(F)_\lambda.$$

\begin{theorem}
 Let $F \in {\rm GL}(n,k)$ be normalizable ($n \geq 2$). There exists a right $H(F)$-comodule algebra structure  $\rho : R(F) \to R(F) \otimes H(F)$ on $R(F)$ defined by
\[
\rho(x_i)=\sum_{k=1}^n x_k \otimes u_{ki}, \quad \rho(y_i) = \sum_{k=1}^n y_k \otimes v_{ki}.
\]
Moreover for any $\lambda \in \mathbb N * \mathbb N$, $R(F)_\lambda$ is an $H(F)$-subcomodule, and contains a unique simple $H(F)$-subcomodule, denoted $L(\lambda)$. We get in this way a bijection $\mathbb N * \mathbb N \simeq {\rm Irr}(H(F))$, $\lambda \mapsto [L(\lambda)]$. 
\end{theorem}

\begin{proof}
The existence of the announced comodule algebra structure is a straightforward verification, and it is also immediate that each $R(F)_\lambda$ is a subcomodule. Similarly one checks the existence of an $H(F)$-comodule algebra map
$\Psi : R(F) \to H(F)$ defined by
\[
\Psi(x_i)=u_{1i}, \quad \Psi(y_i)= v_{ni}.
\]
If $F$ is diagonal, one checks, using the diamond Lemma as in \cite{bi2}, that $\Psi$ is injective.
We now restrict for a moment to the case $H(F)=H(q)$ (and we put $R(q)=R(F)$). We have $\Psi(R(q)_\alpha) \subset {\rm Ind}_{B(q)}^{H(q)}(\alpha)$ and  $\Psi(R(q)_\beta) \subset {\rm Ind}_{B(q)}^{H(q)}(\beta)$, and more generally
 $\Psi(R(q)_\lambda) \subset {\rm Ind}_{B(q)}^{H(q)}(\lambda)$ for any $\lambda \in \mathbb N * \mathbb N$.
Since $\Psi$ is an injective morphism of comodules, it follows from Theorem 7.1 and Proposition \ref{pw} that 
$R(q)_\lambda$ contains a unique simple comodule, isomorphic to the simple comodule $L(\lambda)$ of Proposition \ref{pw}. The last assertion for $H(q)$ follows from Theorem \ref{remain}. 

Back to general case, let $q \in k^\times$ be such that the tensor categories of comodules over $H(F)$ and $H(q)$ are equivalent \cite{bi2}. Using the techniques of \cite[Section 5]{bicogro}, we leave it to the  reader to check that the tensor equivalence transforms the $H(F)$-comodule algebra $R(F)$ into the $H(q)$-comodule algebra $R(q)$, with preservation of the grading, and hence the case of $H(q)$ concludes the proof.
\end{proof}

\subsection{Grothendieck ring}

In this subsection we explain how some results of Chirvasitu \cite{chi} on the structure of the Grothendieck ring $K(H(F))$ (where $F$ is normalizable) can also be derived from our results. As explained in \S\ref{ss:cosovereign-big-cell} it is sufficient to consider the case of $H(q)$ for $q \in k^\times$.

Theorem \ref{bigfree} proves that the simple $H(q)$-comodules are parametrized by $\mathbb{N} * \mathbb{N}$; but its proof also provides a concrete description of these simple comodules. In fact, using the notation of this proof, if $\lambda \in \mathbb{N} * \mathbb{N}=(\mathbb{F}_2)_+$, the restriction of the simple comodule $L(\lambda)$ to $\mathcal{O}(\mathrm{SL}_q(2)) * k[z,z^{-1}]$ under the embedding $\iota$ is the simple $\mathcal{O}(\mathrm{SL}_q(2)) * k[z,z^{-1}]$-comodule $L(\iota'(\lambda))$, which itself is a simple alternated comodule (see Proposition \ref{freesl2} and its proof), i.e.~an ``alternating'' tensor product of simple comodules for $\mathcal{O}(\mathrm{SL}_q(2))$ and $k[z,z^{-1}]$. For instance, the restriction of $L(\alpha)$ is $k_z \otimes L(1)$, 
the restriction of $L(\beta)$ is $L(1) \otimes k_{z^{-1}}$, and the restriction of $L(\beta \alpha)$ is $L(2)$.

Recall the results of \S\ref{ss:Groth-ring-free-product} in the special case $\Gamma=\mathbb{Z}$. The morphism
\begin{equation}
\label{eqn:Grothendieck-restriction}
K \bigl( H(q) \bigr) \to K\bigl( \mathcal{O}(\mathrm{SL}_q(2)) * k[z,z^{-1}] \bigr)
\end{equation}
induced by restriction is injective, and identifies $K(H(q))$ with the submodule spanned by the classes of simple modules $L(\iota'(\lambda))$ for $\lambda \in (\mathbb{F}_2)_+$.

\begin{lemma}
\label{lem:Grothendieck-Hq}
The image of $K(H(q))$ in $K\bigl( \mathcal{O}(\mathrm{SL}_q(2)) * k[z,z^{-1}] \bigr)$ is the subalgebra generated by the classes of $k_z \otimes L(1)$ and $L(1) \otimes k_{z^{-1}}$.
\end{lemma}

\begin{proof}
As \eqref{eqn:Grothendieck-restriction} is an algebra morphism, the image of $K(H(q))$ contains the subalgebra generated by $[k_z \otimes L(1)]=[L(\iota'(\alpha))]$ and $[L(1) \otimes k_{z^{-1}}]=[L(\iota'(\beta))]$. 

Let us now prove the reverse inclusion. Recall that the simple $\mathcal{O}(\mathrm{SL}_q(2)) * k[z,z^{-1}]$-comodules are parametrized by the submonoid of the free product $\mathbb{Z} * \mathbb{Z}$ (with generators $t$ and $z$) generated by $t$, $z$ and $z^{-1}$. If $\lambda$ is in this monoid, we denote by $\ell(\lambda) \in \mathbb{N}$ its total degree in $t$. Then $\ell$ satisfies $\ell(\lambda \mu)=\ell(\lambda) + \ell(\mu)$. We will prove by induction on $n$ that if $\ell(\iota'(\lambda)) \leq n$ then $[L(\iota'(\lambda))]$ is in the subalgebra generated by $[k_z \otimes L(1)]$ and $[L(1) \otimes k_{z^{-1}}]$. In fact, if $\ell(\iota'(\lambda)) >0$ then $\lambda$ can be written as $\mu \alpha$ or $\mu \beta$ for some $\mu$ with $\ell(\iota'(\mu)) = \ell(\iota'(\lambda))-1$. To fix notation, consider the first case. Using the fact that $L(\iota'(\mu))$ is a simple alternated comodule and equation \eqref{eqn:tensor-product-sl2}, we know that
\[
[L(\iota'(\lambda))] - [L(\iota'(\mu))] \cdot [L(\iota'(\alpha))]
\]
is a linear combination of classes $[L(\iota'(\nu))]$ with $\ell(\iota'(\nu)) < \ell(\iota'(\lambda))$.
Hence by induction this element is in the subalgebra generated by $[k_z \otimes L(1)]$ and $[L(1) \otimes k_{z^{-1}}]$. Similarly, by induction $[L(\iota'(\mu))]$ is in this subalgebra, and we deduce the same property for $[L(\iota'(\lambda))]$.
\end{proof}

Let $\mathbb{Z} \langle X,Y \rangle$ be the free ring generated by two variables $X,Y$.
Consider the ring morphism
\[
\psi : \mathbb{Z} \langle X,Y \rangle \to K(H(q))
\]
sending $X$ to $[L(\alpha)]$ and $Y$ to $[L(\beta)]$. Then we have a commutative diagram
\[
\xymatrix@C=2cm{
\mathbb{Z} \langle X,Y \rangle \ar[r]^-{\psi} \ar[d] & K \bigl( H(q) \bigr) \ar[d]^-{\eqref{eqn:Grothendieck-restriction}} \\
\mathbb{Z}[T] * \mathbb{Z}[z,z^{-1}] \ar[r]^-{\varphi}_-{\sim} & K\bigl( \mathcal{O}(\mathrm{SL}_q(2)) * k[z,z^{-1}] \bigr),
}
\]
where $\varphi$ is defined in \S\ref{ss:Groth-ring-free-product} and the left vertical morphism sends $X$ to $zT$ and $Y$ to $Tz^{-1}$. As the latter morphism is injective and $\varphi$ is an isomorphism we obtain that $\psi$ is injective. By Lemma \ref{lem:Grothendieck-Hq} this morphism is also surjective, hence it is an isomorphism. We have proved 
the following result, which generalizes \cite[Corollary 5.5]{bi2} and was first obtained in \cite[Corollary 1.2]{chi}.

\begin{proposition}
The morphism $\psi$ is a ring isomorphism $\mathbb{Z} \langle X,Y \rangle \xrightarrow{\sim} K(H(q))$.
\end{proposition}

In fact our approach also determines the structure of the Grothendieck semi-ring $K_+(H(q))$ of the tensor category of finite dimensional $H(q)$-comodules (at least in theory), even in the case $H(q)$ is not co-semisimple. Indeed, assume for simplicity that $q$ is a primitive root of unity of odd order $N > 1$. Then if $k=nN+m$ with $n,m \geq 0$ and $m<N$ by \cite[Theorem 9.4.1]{pw} we have
\[
L(k) \cong L(m) \otimes L(n)^{(1)},
\]
where $(\cdot)^{(1)}$ denotes the Frobenius twist. Hence to describe tensor products of simple $\mathcal{O}(\mathrm{SL}_q(2))$-comodules it is sufficient to describe tensor products $L(m) \otimes L(m')$ where $m,m' < N$ and $L(n)^{(1)} \otimes L(n')^{(1)}$. However the former are uniquely determined by the rule
\[
L(m) \otimes L(1) = \begin{cases}
L(m+1) \oplus L(m-1) & \text{if } m<N-1 \\
L(1)^{(1)} \oplus 2L(N-2) & \text{if } m=N-1
\end{cases}
\]
and the later are uniquely determined by the rule
\[
L(n)^{(1)} \otimes L(1)^{(1)} \cong L(n+1)^{(1)} \oplus L(n-1)^{(1)}.
\]
These rules determine the Grothendieck semi-ring $K_+ \bigl( \mathcal{O}(\mathrm{SL}_q(2)) \bigr)$, and then it is easy to deduce the structure of the semi-ring $K_+ \bigl( \mathcal{O}(\mathrm{SL}_q(2)) * k[z,z^{-1}] \bigr)$ since every simple comodule is a simple alternated comodule. Finally one can deduce the structure of $K_+(H(q))$ since it is the semi-group generated by $[k_z \otimes L(1)]$ and $[L(1) \otimes k_{z^{-1}}]$ in $K_+ \bigl( \mathcal{O}(\mathrm{SL}_q(2)) * k[z,z^{-1}] \bigr)$. 


\section{2-cocycle deformations}

In this section we examine the behaviour of dense big cells under a 2-cocycle deformation (dual of Drinfeld twist).
We remark that the 2-cocycle deformation of a Hopf algebra having a dense big cell still has a dense big cell under the assumption that the 2-cocycle is induced by a 2-cocycle on the weight group. On the other hand we study an example that shows this is not  true if one removes the assumption.

\subsection{Deformation induced by a 2-cocycle on $\Lambda$}

Let $H$ be a Hopf algebra.  Recall (see e.g.~\cite{do})
that a 2-cocycle on $H$ is a convolution invertible linear map
$\sigma : H \otimes H \longrightarrow k$ satisfying
$$\sigma(x_{(1)}, y_{(1)}) \sigma(x_{(2)}y_{(2)},z) =
\sigma(y_{(1)},z_{(1)}) \sigma(x,y_{(2)} z_{(2)})$$
and $\sigma(x,1) = \sigma(1,x) = \varepsilon(x)$, for $x,y,z \in H$. 
As an example, if $H=k\Gamma$ is a group algebra, any ordinary 2-cocycle $\sigma \in Z^{2}(\Gamma,k^\times)$ extends by linearity to a 2-cocycle on the Hopf algebra $k\Gamma$.

To  a 2-cocycle $\sigma$ on $H$, one associates \cite{do} the 
Hopf algebra $H^{\sigma}$ defined as follows.  
As a coalgebra $H^{\sigma} 
= H$. The product 
in $H^{\sigma}$ is defined by
$$[x] [y]= \sigma(x_{(1)}, y_{(1)})
\sigma^{-1}(x_{(3)}, y_{(3)}) [x_{(2)} y_{(2)}]
\qquad \text{for } x,y \in H,$$
where an element $x \in H$ is denoted $[x]$, when viewed as an element 
of $H^{\sigma}$. See \cite{do} for the formula defining the antipode.  The Hopf algebras $H$ and $H^\sigma$ have equivalent tensor categories of comodules, see e.g.~\cite{sc1}.

Let $f : H \to L$ be a Hopf algebra map and let
$\sigma : L \otimes L \to k$ be a  2-cocycle on $L$.
Then $\sigma_{f} = \sigma \circ (f \otimes f) : H \otimes H \to k$ is a  2-cocycle.
In what follows the cocycle $\sigma_f$ will simply be denoted by $\sigma$; this should not cause any confusion.

The following result provides new examples of Hopf algebras having a dense big cell. The proof is immediate.

\begin{proposition}\label{cocybig}
 Let $H$ be a Hopf algebra having a dense big cell $(B,B', \Lambda)$ and let $\sigma \in Z^2(\Lambda,k^\times)$. Then $(B^{\sigma}, B'^{\sigma}, \Lambda)$ is a dense big cell for $H^\sigma$.
\end{proposition}

As an example, one can deduce from the proposition that the multiparametric quantum ${\rm GL}(n)$ from \cite{ast} has a dense big cell. (This can also be checked directly, following the verification in \cite{pw}, as in \cite{gar}.)

\subsection{Example of the Jordanian quantum ${\rm SL}(2)$}

Let us now study an example of a Hopf algebra that does not have a dense big cell. The Hopf algebra in question is $\mathcal O({\rm SL}(2)_J)$, corresponding to the so-called Jordanian quantum ${\rm SL}(2)$. This Hopf algebra arose in several independent papers (see \cite{dvl, dmmz, gu, z}; we do not claim that the list is exhaustive), and is known to be    
a 2-cocycle deformation of $\mathcal O({\rm SL}(2))$, see e.g.~\cite{eg}. In particular, its tensor category of comodules is equivalent to that of $\mathcal O({\rm SL}(2))$. We assume from now on that the base field has characteristic different from 2.

Recall that $\mathcal O({\rm SL}(2)_J)$ is the algebra generated by $a, b, c, d$, with relations
\begin{gather*}
ca-ac=c^2 = cd-dc, \quad ba-ab=1-a^2, \quad bd-db= 1-d^2 \\ 
ad-da=ac-dc, \quad cb -bc =ac+cd, \quad 1=ad-bc-ac.
\end{gather*}
The formulas for its comultiplication and counit are the same as those for $\mathcal O({\rm SL}(2))$, while the antipode is defined by
$$S\begin{pmatrix}
    a & b \\
c & d
   \end{pmatrix} = \begin{pmatrix}
d-c & a-b+c-d  \\
-c & a+c \end{pmatrix}.$$

We begin with the following easy lemma.

\begin{lemma}
Let $L$ be a Hopf algebra with involutive antipode and let $f : \mathcal O(SL(2)_J) \rightarrow L$ be a Hopf algebra map. Then $f(c)=0$, $f(a)=f(d)$ and $f(a)^2=1$. 
\end{lemma}

\begin{proof} We have $$\begin{pmatrix}
    f(a) & f(b) \\
f(c) & f(d)
   \end{pmatrix} =
 S^2\begin{pmatrix}
    f(a) & f(b) \\
f(c) & f(d)
   \end{pmatrix} = \begin{pmatrix}
f(a)+2f(c) & -2f(a)+f(b)-4f(c)+2f(d)  \\
f(c) & -2f(c)+f(d) \end{pmatrix}.$$
It follows that $f(c)=0$ and  that $f(a)=f(d)$. The identity $f(a)^2=1$ then follows from $1=ad-bc-ac$.
\end{proof}

\begin{proposition}
 The Hopf algebra $\mathcal O(SL(2)_J)$  does not have a dense big cell.
\end{proposition}

\begin{proof}
 Assume that $H=\mathcal O(SL(2)_J)$ has a dense big cell. In particular, by Theorem \ref{remain}, there exists a group $\Gamma$, a surjective Hopf algebra map $f : H \rightarrow k\Gamma$ and an injective map ${\rm Irr}(H) \hookrightarrow \Gamma$. The lemma and Proposition \ref{pointed} then ensure that $\Gamma \simeq {\rm Gr}(k\Gamma)$ has at most two elements, which contradicts the existence of the injection ${\rm Irr}(H) \hookrightarrow \Gamma$.
\end{proof}

\end{document}